\documentclass[a4paper,10pt]{article}
\usepackage[utf8]{inputenc}

\usepackage{graphicx}
\usepackage[hidelinks]{hyperref}
\usepackage{amsthm,amssymb,amsmath}
\author{Matthieu Rosenfeld\\
\small \it LIRMM, CNRS, Universit\'e de Montpellier}

\newtheorem{theorem}{Theorem}
\newtheorem*{theorem*}{Theorem}
\newtheorem{lemma}[theorem]{Lemma}

\newtheorem{fact}[theorem]{Fact}
\newtheorem{corollary}[theorem]{Corollary}

\theoremstyle{remark}

\newcommand{\A}{\mathcal{A}}
\newcommand{\F}{\mathcal{F}}
\title{Finding lower bounds on the growth and entropy of subshifts over countable groups}

\begin{document}
\maketitle
\begin{abstract}
We provide a lower bound on the growth of a subshift based on a simple condition on the set of forbidden patterns defining that subshift. Aubrun et Al. showed a similar result based on the Lov\'asz Local Lemma for subshifts over any countable group, and Bernshteyn extended their approach to deduce some lower bound on the exponential growth of the subshift. Our result has a simpler proof, is easier to use for applications, and provides better bounds on the applications from their articles (although it is not clear that our result is strictly stronger in general).

In the particular case of subshifts over $\mathbb{Z}$, Miller gave a similar, but weaker condition that implied the nonemptiness of the associated shift. Pavlov used the same approach to provide a condition that implied exponential growth. We provide a version of our result for this particular setting, and it is provably strictly stronger than the result of Pavlov and the result of Miller. In practice, it leads to considerable improvement in the applications.

We also apply our two results to a few different problems including strongly aperiodic subshifts, nonrepetitive subshifts, and Kolmogorov complexity of subshifts.
\end{abstract}

\section{Introduction}
Let $\A$ be a finite alphabet and $(G,.)$ be a group. The elements of $\A$ are called \emph{letters}. A \emph{configuration} is an element of the set $\A^G=\{x:G\rightarrow \A\}$ (which we can see as a coloring of the elements of $G$ by $A$). A \emph{support} is a non-empty finite set $S\subseteq G$ and a \emph{pattern with support $S$} is an element of $\A^S$.
A pattern $p\in \A^S$ \emph{appears} in a configuration $x\in\A^G$ if there exists $g\in G$ such that for all $s\in S$, $x(gs)=p(s)$. We then say that $p$ appears in $X$ at \emph{position} $g$. If $p$ does not appear in $x$, then $x$ \emph{avoids} $p$.
For any set of patterns $\F$, $X_\F$, is the set of configurations avoiding  $\F$, that is
\begin{equation*}
X_{\F}=\{x\in \A^G: \forall f\in\F, x\text{ avoids } f\}\,.
\end{equation*}
A \emph{subshift} $X\subseteq \A^G$ is a set of configurations defined by a set of forbidden patterns, that is, $X$ is a subshift if there exists a set of patterns $\F$ such that $X= X_\F$ ($\F$ is not necessarily unique).

It seems natural that if the set of forbidden patterns is ``small'' enough, then $X_\F$ should be non-empty regardless of what exactly is in $\F$. Intuitively, it is easier to forbid a few patterns than many patterns and it is easier to forbid large patterns than small patterns. Aubrun, Barbieri, and Thomassé gave a sufficient condition for a subshift to be non-empty and this condition only depends on the sizes of the different forbidden patterns \cite{abb2019}. Their result relies on the Lovász Local Lemma. Bernshteyn extended on their idea and showed in a larger but similar context that similar conditions even implied some lower-bounds on the ``size'' of the subshift \cite{Bernshteyn}. He considered five different notions of ``size of a subshift'', one of them being the so-called topological entropy of a subshift. Intuitively, the topological entropy of a subshift describes the growth of the number of orbits in the associated dynamical system. From a more combinatorial point of view, the complexity of a subshift is the function that maps $n$ to the number of different patterns that appear in the subshift by translation of a given support of size $n$ and the entropy is the logarithm of the growth rate of the complexity.

If instead of any group we restrict our attention to $G=\mathbb{Z}$, we can use the structure of $\mathbb{Z}$ to obtain more precise results. In this setting, Pavlov gave a condition on $\F$ that implies lower bounds on the entropy of $X_\F$ \cite{pavlov}. His proof was based on an idea introduced by Miller to show that a subshift is nonempty \cite{MillerSubshift}. He then used this condition to deduce a condition for other interesting properties of subshifts such as the uniqueness of measures of maximal entropy.

In this article, we give a general lower bound on the entropy of $X_\F$ that only depends on the size of the patterns of $\F$. Our condition is hard to compare to the bounds of \cite{abb2019} and \cite{Bernshteyn}, but it seems to give better results in general and to be easier to optimize. The authors of \cite{abb2019}, use their criterion to prove the existence of nonempty aperiodic shifts over any countable group and we strengthen their result by showing that there exist strongly aperiodic subshifts with entropy arbitrarily close to the entropy of the full-shift. We also deduce that there exist nonrepetitive colorings of the Cayley graph of a group using a much smaller alphabet (by a factor $\approx 2^{17}$) than the one given in \cite{abb2019}.

We then focus our attention on the special case $G=\mathbb{Z}$ and we provide a condition that is strictly better than the one given in \cite{pavlov}. In practice, the improvement on the resulting bounds seems to be considerable.

There are two main ingredients to our proof. The first one is to show that if the groups is countable and amenable, then the growth rate of the number of patterns that avoids $\F$ is the same as the growth rate of the number of patterns that appear in a configuration of $X_\F$. It seems to be folklore, but we were not able to find this result in the literature. Shur showed a version of this result restricted to the case $G=\mathbb{Z}$ \cite{shurcomplexity}. His proof relied on automata and regular languages, while our proof only relies on combinatorics and topology. The same result was also proven over $\mathbb{Z}^d$, but for the specific case of subshifts of finite type \cite{FriedlandZd, HochmanZd}. We then provide some lower bounds on the number of patterns avoiding $\F$, which implies a lower bound on the entropy of the associated subshift.

The argument used for our lower bound is a simple counting technique recently introduced in \cite{rosenfeldCounting} and already used in a few different settings \cite{BernshteynTriangleFree,pirotTriangleFree,Thuelist,wanlessWood}. In the setting of combinatorics on words, a similar technique was already known under the name \emph{power series method} \cite{BELL20071295,BLANCHETSADRI201317,doublepat,rampersadpowerseries}. The conditions in \cite{MillerSubshift,pavlov} resemble the conditions obtained by the power series methods, but it appears that the link was never established. In fact, if $G=\mathbb{Z}$, Lemma \ref{lemmacountingZ} can also be deduced as a particular case of the conditions from \cite{doublepat}.

The article is organized as follows. We first provide a few definitions and notations. Then we prove in Section \ref{growthsec} that, for countable amenable groups, the growth rate of locally admissible patterns is the same as the growth rate of globally admissible patterns. We then give our lower bound on the growth rate in the general case followed by an application of this bound. We conclude with our bound in the particular setting of $G=\mathbb{Z}$ and a comparison with the bound given by Pavlov.

\section{Definitions and notations}
For any set $S$ and any integer $n$, $\binom{S}{n}$ is the set of all subsets of $S$ of size $n$. For any function $f:A\rightarrow B$ and $A'\subseteq A$, the \emph{restriction of $f$ to $A'$} is the function $f\big|_{A'}: A'\rightarrow B$ such that for all $a\in A'$, $f\big|_{A'}(a)=f(a)$.
Given a set of functions $C\subseteq B^A$ and a subset $A'\subset A$, we write $C\big|_{A'}=\{c\big|_{A'} : c\in C\}$. For any two sets $A$ and $B$ of elements of a group $(G,\cdot)$, we write $A\cdot B= \{a\cdot b: a\in A, b\in B\}$ and $A^{-1}=\{a^{-1}:a\in A\}$.

The support of a pattern $p$ is given by $\operatorname{supp}(p)=S$.
The size of a pattern is defined to be equal to the size of its support, that is, $|p|= |\operatorname{supp}(p)|$.
The set of non-empty patterns over $G$ is
\begin{equation*}
\A^+_G=\bigcup_{\substack{S\subset G\\0< |S|<\infty}}\A^S\,.
\end{equation*}

For any set of patterns $\F\subseteq\A^+_G$, we say that a pattern $p$ is \emph{globally admissible} if there exists $x\in X_\F$ such that $p$ appears in $x$. It is \emph{locally admissible} if $p$ avoids all patterns $f\in \F$. For all support $S\subset G$ and set of forbidden patterns $\F\subseteq\A^+_G$, we let $\mathcal{G}_\F^{(S)}$ and $\mathcal{L}_\F^{(S)}$ be respectively the set of globally and locally admissible patterns of support $S$. By definition, every globally admissible pattern is necessary locally admissible, that is, $\mathcal{G}_\F^{(S)}\subseteq \mathcal{L}_\F^{(S)}$ for all $S$. For any set of forbidden patterns  $\F\subseteq\A^+_G$ and integer $n$, we define the \emph{global complexity} of  $X_\F$ as
\begin{equation*}
\mathbf{G}^{(n)}_\F= \min_{S\in \binom{G}{n}}\left|\mathcal{G}_\F^{(S)}\right|\,,
\end{equation*}
and the \emph{local complexity} as
\begin{equation*}
\mathbf{L}^{(n)}_\F= \min_{S\in \binom{G}{n}}\left|\mathcal{L}_\F^{(S)}\right|\,.
\end{equation*}
The \emph{global growth} of $X_\F$ is given by
\begin{equation*}
\alpha(\F)=\lim_{n\rightarrow \infty} \left(\mathbf{G}_\F^{(n)}\right)^{1/n}\,,
\end{equation*}
and the \emph{local growth} is
\begin{equation*}
\tilde\alpha(\F)=\lim_{n\rightarrow \infty}\left(\mathbf{L}_\F^{(n)}\right)^{1/n}\,.
\end{equation*}
Let us justify that the local and the global growth are well-defined.
\begin{lemma}\label{Feketeapp}
For any $\F\subseteq\A^+_G$, the quantities $\alpha(\F)$ and $\tilde\alpha(\F)$ are well-defined. Moreover, for all $n\ge0$,
  \begin{equation*}
\left(\mathbf{G}_\F^{(n)}\right)^{1/n}\ge \alpha(\F)\,,\quad\text{ and } \quad
\left(\mathbf{L}_\F^{(n)}\right)^{1/n}\ge \tilde\alpha(\F)\,. 
  \end{equation*}
\end{lemma}
\begin{proof}
Let $i,j\in\mathbb{N}$ and $S_i\in \binom{G}{i}$ and $S_j\in \binom{G}{j}$ be such that $\mathbf{G}^{(i)}_\F= \left|\mathcal{G}_\F^{(S_i)}\right|$ and $\mathbf{G}^{(j)}_\F= \left|\mathcal{G}_\F^{(S_j)}\right|$. We can assume that $S_i\cap S_j=\emptyset$ since these two sets are finite and for all $S$ the value of $\mathcal{G}_\F^{(S)}$ is invariant by translation of $S$. 
We have  
\begin{equation*}
\left|\mathcal{G}_\F^{(S_j\cup S_i)}\right|\le  \left|\mathcal{G}_\F^{(S_i)}\right|\cdot  \left|\mathcal{G}_\F^{(S_j)}\right|=\mathbf{G}^{(i)}_\F \cdot \mathbf{G}^{(j)}_\F 
\end{equation*}
which implies
\begin{equation*}
  \mathbf{G}^{(i+j)}_\F \le\mathbf{G}^{(i)}_\F \cdot \mathbf{G}^{(j)}_\F \,.
\end{equation*}
The sequence $\left(\mathbf{G}_\F^{(n)}\right)_{n\ge0}$ is submultiplicative. Our result is then a direct consequence of Fekete's Lemma.   
The same argument holds for $\left(\mathbf{L}_\F^{(n)}\right)_{n\ge0}$.
\end{proof}

For any subshift $X$, we let the \emph{growth of the subshift $X$} be $\alpha(X)=\alpha(X_\F)$ for any $\F$ such that $X=X_\F$.
By definition, the global complexity doesn't depend on the choice of $\F$. On the other hand, by definition, the local complexity depends on the choice of $\F$ (and given a subshift $X$ the $\F$ such that $X= X_\F$ is not necessarily unique). However, in the next section we show that if $G$ is amenable, for any set $\F$  the two associated growths are identical, that is, $\alpha(\F)=\tilde\alpha(\F)$. Thus for any subshift $X$, we can choose the most convenient $\F$ such that $X=X_\F$ and compute $\alpha(X)=\alpha(\F)=\tilde\alpha(\F)$. We will provide in Section \ref{generalbound} a way to lower bound $\tilde\alpha(\F)$ under some conditions on $\F$.

\subsection{Amenable groups and topological entropy}
We say that a sequence $(F_n)_{i\in\mathbb{N}}$ is a \emph{symmetric Følner sequence} of $G$ if
\begin{enumerate}
\item for all $i$, $F_i$ is a finite subset of $G$,
\item for all $g\in G$, $\lim\limits_{i\rightarrow\infty} \frac{|(g \cdot F_i)\Delta F_i|}{|F_i|}=0$,
\item for all $i$, $F_i= F_i ^{-1}$.
\end{enumerate}
Remember that a countable group is amenable if and only if it admits a symmetric Følner sequence.
Conditions  1. and 2. define a Følner sequence, and the existence of a Følner sequence is equivalent to $G$ being amenable.
The fact that, this is still equivalent when adding condition 3. is proven in \cite[Corollary 5.3]{symmetricFolner} and it will be more convenient for us. One easily deduces from 2. and 3. that for all finite set $S$,

\begin{equation*}
\lim\limits_{i\rightarrow\infty} \frac{|(S \cdot F_i)\Delta F_i|}{|F_i|}=0 = \lim\limits_{i\rightarrow\infty} \frac{|(F_i\cdot S)\Delta F_i|}{|F_i|}\,.
\end{equation*}

If $G$ is a countable amenable group, and $(S_n)_{n\ge0}$ is a Følner sequence, then the so-called \emph{topological entropy} of $X_\F$ is given by $h(X_\F)=\lim\limits_{n\rightarrow \infty}\frac{\log \mathcal{G}_\F^{(S_n)}}{|S_n|}$. As a direct consequence of the definition, for any subshift $X_\F$,
\begin{equation*}
h(X_\F)\ge\log \alpha(X_F)\,.
\end{equation*}
This lower bound is sufficient for our applications, since we only provide lower bounds on $\alpha(X_F)$ from which we deduce lower bounds $h(X_\F)$. Let us however recall that the equality holds between these quantities, that is,
\begin{equation*}
h(X_\F)=\log \alpha(X_F)\,.
\end{equation*}
This can be deduced from \cite[Proposition 3.3]{entropyandShearer} and from the fact that
\begin{equation*}
\alpha(X_F)=\inf_{S\subseteq G, S \text{ finite}}\left|\mathcal{G}_\F^{(S)}\right|^{1/|S|}\,
\end{equation*}
which is itself a consequence of Fekete's Lemma applied to the submultiplicativity of $\left(\mathbf{G}_\F^{(n)}\right)_{n\ge0}$.

\subsection{A combinatorial version of Shearer's inequality}
We provide here a combinatorial version of Shearer's inequality. This Lemma will be crucial in proving that the local growth and the local growth are equal over an amenable countable group. This is heavily inspired by \cite{entropyandShearer}.

Given a set $S$ and a collection of subsets  $S_1, ..., S_n\subseteq S$, we say that $S_1, ..., S_n$ is an $r$-cover of $S$ if for all $s\in S$, $s$ appears in at least $r$ of the subsets, that is,
$\forall s\in S, |\{i\in\{1,\ldots, n\}: s\in S_i\}|\ge r$ (some of the $S_i$ could be pairwise identical in which case they are counted with their multiplicity).

\begin{theorem}\label{CombinatorialShearer}
  Let $C\subseteq \mathcal{A}^{S}$ be a set of colorings of a finite set $S$ by a finite alphabet $\mathcal{A}$. Let $S_1, ..., S_t$ be an $r$-cover of $S$, then
  \begin{equation*}
  |C|\le \left(\prod_{i=1}^{t}\left|C\big|_{S_i}\right|\right)^{1/r}\,.
  \end{equation*}
\end{theorem}

The proof is a direct application of Shearer's inequality for entropy.
Remember that the entropy of a discrete random variable $X$ taking value over a set $\mathcal{X}$ and is distributed according to $p:\mathcal{X}\rightarrow[0,1]$ is given by
\begin{equation*}
H[X]=-\sum_{x\in\mathcal{X}}p(x)\log p(x)\,.
\end{equation*}

\begin{theorem}[Shearer's inequality]
  If $X_1, ..., X_d$ are random variables and $S_1, ..., S_n$ is an $r$-cover of $\{1, 2, \ldots, d\}$, then
  \begin{equation*}
H\left[(X_1,\ldots, X_d)\right]\le\frac{1}{r}\sum_{i=1}^{n}H\left[(X_j)_{j\in S_i}\right]\,.
  \end{equation*}
\end{theorem}

The following fact is a direct consequence of Jensen's inequality.
\begin{fact}\label{entropySizeSet}
Let $X$ be a discrete random variable taking value over a set $\mathcal{X}$ distributed according to $p:\mathcal{X}\rightarrow[0,1]$, then
\begin{equation*}  
H[X]\le \log |\mathcal{X}|
\end{equation*}
with equality if the distribution is uniform.
\end{fact}

We are now ready to prove our combinatorial version of Shearer's inequality.
\begin{proof}[Proof of Theorem \ref{CombinatorialShearer}]
 Let $\mathbf{C}$ be a random variable taking value over $C$ uniformly at random. For all $\sigma\in S$, let $\mathbf{X}_\sigma$ be the random variable such that $\mathbf{X}_\sigma=\mathbf{C}(\sigma)$. Shearer's inequality implies
\begin{equation*}
H[\mathbf{C}] \le\frac{1}{r}\sum_{i=1}^{t}H\left[(\mathbf{X}_\sigma)_{\sigma\in S_i}\right]\,.
\end{equation*}

Since for all $i$, $(\mathbf{X}_\sigma)_{\sigma\in S_i}$ takes value over $C\big|_{S_i}$, we can apply Fact \ref{entropySizeSet} to deduce 
\begin{equation*}
H\left[(\mathbf{X}_\sigma)_{\sigma\in S_i}\right]\le \log\left|C\big|_{S_i}\right|\,.
\end{equation*}

By Fact \ref{entropySizeSet}, we also have $\log |C|= H[\mathbf{C}]$. 
These three equations imply
\begin{equation*}
\log |C|\le\frac{1}{r}\sum_{i=1}^{t}\log\left|C\big|_{S_i}\right|\,,
\end{equation*}
which,  by removing the $\log$, implies the desired inequality.
\end{proof}

\section{Growth of locally admissible and globally admissible patterns in amenable groups}\label{growthsec}
This section is devoted to the proof of the following ``folklore'' result.
\begin{theorem}\label{extendablesamegrowth}
  For any countable amenable group $G$ and $\F\subseteq\A^+_G$, we have
  \begin{equation*}\alpha(\F)=\tilde\alpha(\F)\,.\end{equation*}
\end{theorem}
For this proof, we use a notion of $t$-extendable pattern.
Let $(g_i)_{i> 0}$ be any enumeration of $G$, that is, the $g_i$ are pairwise distinct and $G=\{g_i: i> 0\}$. For all $i>0$, we let $G_i=\{g_1,g_2\ldots,g_i \}$. For any integer $t>0$, we say that a locally admissible pattern $p\in \mathcal{L}^{(S)}_\F$ is \emph{$t$-extendable} if there exists a pattern $p'\in \mathcal{L}^{(S\cup S\cdot G_t)}_\F$ such that $p'\big|_{S}=p$. In other words, $p$ is $t$-extendable if whenever we add $\operatorname{supp}(p)\cdot G_t$ to the support of $p$, we can extend $p$ into a larger locally admissible pattern with this new support.
For any integer $t>0$ and finite set $S\subset G$, we let $E^{(S)}_t$ be the set of patterns with support $S$ that are $t$-extendable.

This notion of  $t$-extendability has two interesting properties for us.
First, by compactness, if a pattern is $t$-extendable for all $t$, then it is globally admissible which will be useful to relate the growth of $t$-extendable patterns to the growth of globally admissible patterns. Second, $t$-extendability is ``preserved by translation'', as illustrated by the following claim, which will be useful to relate the growth of $t$-extendable patterns to the growth of locally admissible patterns.
\begin{fact}\label{extendabilityIsTranslationInvariant}
For all $S\subseteq G$, $g\in G$ and $t>0$, we have the following equality
\begin{equation*}
\left|E^{(S)}_t\right| =\left|E^{(g\cdot S)}_t\right|\,.
\end{equation*}
\end{fact}
\begin{proof}
Fix $t>0$. It is enough to prove $\left|E^{(S)}_t\right|\le\left|E^{(g\cdot S)}_t\right|$ for all $S\subseteq G$ and $g\in G$.
The other direction is a direct consequence of the same inequality applied to $g\cdot S$ and $g^{-1}$.

Consider a $t$-extendable coloring $c\in E^{(S)}_t$ that can be extended to a locally admissible coloring $c'$ of $S\cup S\cdot G_t$. Let $d:g\cdot S\rightarrow\mathcal{A}$ and $d':(g\cdot S)\cup(g\cdot S\cdot G_t)\rightarrow\mathcal{A}$ be the colorings such that:
\begin{itemize}
  \item for all $x\in g\cdot S$, $d(x)= c(g^{-1}x)$,
  \item  for all $x\in (g\cdot S)\cup(g\cdot S\cdot G_t)$, $d'(x)= c'(g^{-1}x)$.
\end{itemize}
For all $x\in g\cdot S$, we have by construction $d'(x)= c'(g^{-1}x)= c(g^{-1}x)=d(x)$, that is, $d= d'\big|_{g\cdot S}$.
For the sake of contradiction, suppose that $d'$ is not locally admissible.
Then there exists a forbidden pattern $p$ that appears in $d'$.
That is, there exists $f\in G$ such that $f\cdot \operatorname{supp}(p)\subseteq(g\cdot S)\cup(g\cdot S\cdot G_t)$ and
for all $x\in\operatorname{supp}(p)$, $d'(fx)=p(x)$.
This implies that $g^{-1}f\cdot \operatorname{supp}(p)\subseteq S\cup(S\cdot G_t)$ and for all $x\in\operatorname{supp}(p)$, $c'(g^{-1}fx)=d'(fx)=p(x)$.
That is, $c'$ contains an occurrence of the forbidden pattern $p$ which is a contradiction.
Hence, $d$ is a $t$-extendable coloring of $g\cdot S$, that is, $d\in  E^{(g\cdot S)}_t$.
Moreover, each coloring such $c\in E^{(S)}_t$ corresponds to a different $d\in E^{(g\cdot S)}_t$ which implies
\begin{equation*}
\left|E^{(S)}_t\right|\le\left|E^{(g\cdot S)}_t\right|\,
\end{equation*}
as desired.
\end{proof}

We are now ready to prove the main Lemma behind Theorem \ref{extendablesamegrowth}.
\begin{lemma}\label{limToInf}
  For any finite set $S\subseteq G$ and $t\in\mathbb{N}$,
\begin{equation*}
|E^{(S)}_t|\ge (\tilde\alpha(\F))^{|S|}\,.
\end{equation*}
\end{lemma}
\begin{proof}
We first prove that an asymptotic version of this statement holds when we replace $S$ by limits over a symmetric Følner sequence.
We will then use the combinatorial version of Shearer Lemma to extend the result to any set $S$.

Since $G$ is a countable amenable group, it admits a symmetric Følner sequence $(F_i)_{i\ge0}$.
By definition, for any $i$, the restriction to $F_i$ of a locally admissible coloring of $F_i \cup (F_i\cdot G_t)$ yields a $t$-extendable coloring of $F_i$.
Moreover, any $t$-extendable coloring of $F_i$ can be extended in at most $|\mathcal{A}|^{|(F_i\cdot G_t)\setminus F_i|}$ different locally admissible coloring of $F_i \cup (F_i\cdot G_t)$.
We get
\begin{equation*}
\left|E^{(F_i)}_t\right|
\ge \frac{|\mathcal{L}^{(F_i \cup (F_i\cdot G_t))}|}{|\mathcal{A}|^{|(F_i\cdot G_t)\setminus F_i|}}
\ge\frac{|\mathcal{L}^{(F_i)}|}{|\mathcal{A}|^{|(F_i\cdot G_t)\Delta F_i|}}\,.
\end{equation*}
Since $G_t$ is a finite set and $(F_i)_{i\ge0}$ is a symmetric Følner sequence, for all $t>0$,
\begin{equation*}
\lim_{i\rightarrow \infty} \frac{|( F_i\cdot G_t)\Delta F_i|}{|F_i|}=0\,.
\end{equation*}
This implies
\begin{equation}
\lim_{i\rightarrow\infty}\left|E^{(F_i)}_t\right|^{1/|F_i|}\ge\lim_{i\rightarrow\infty}\left|\mathcal{L}^{(F_i)}\right|^{1/|F_i|}
= \tilde\alpha(\F)\,.\label{growthOfFølner}
\end{equation}

We are now ready to prove our Lemma by contradiction.
For the sake of contradiction, suppose that there exists $t>0$, $S\subseteq G$ and $\varepsilon>0$ such that
\begin{equation}\label{contradiction}
|E^{(S)}_t|\le ((1-\varepsilon)\tilde\alpha(\F))^{|S|}\,.
\end{equation}
Let $t'$ be such that $S^{-1}\cdot S\cdot (\{\mathbf{1}_G\}\cup G_t)\subseteq G_{t'}$ (the set on the left-hand part is finite, so there exists such a $t'$).
Let $(X_g)_{g\in F_n S^{-1}}$ be the family such that for all $g\in  F_n S^{-1}$,
\begin{equation*}
X_g = (g\cdot S)\cap F_n\,.
\end{equation*}
The family $(X_g)_{g\in F_n S^{-1}}$ is an $r$-cover of $F_n$ with $r=|S|$.
Indeed, for $f\in F_n$ the condition $f\in g\cdot S$ is equivalent to  $g\in f\cdot S^{-1}$ which is fulfilled by exactly $|S|$ elements of $F_n S^{-1}$.
Theorem \ref{CombinatorialShearer} implies
\begin{equation}\label{eqshearerapplied}
  \left|E_{t'}^{(F_n)}\right|\le \left(\prod_{g\in F_n S^{-1}}\left|E_{t'}^{(F_n)}\big|_{X_g}\right|\right)^{1/|S|}\,.
\end{equation}

Let $c\in E_{t'}^{(F_n)}$ and $g\in F_n S^{-1}$. By definition, $c$ is a locally admissible coloring of $F_n$ that can be extended to a locally admissible coloring of $F_n\cup F_nG_{t'}$. The restriction $c\big|_{X_g}$ is also a restriction of the same locally admissible coloring of $F_n\cup F_nG_{t'}$.
Since  $g\in F_n S^{-1}$,
\begin{equation*}
F_nG_{t'} \supseteq F_nS^{-1}\cdot S\cdot (\{\mathbf{1}_G\}\cup G_t)\supseteq g \cdot S\cdot (\{\mathbf{1}_G\}\cup G_t)= gS\cup gSG_t\,.
\end{equation*}
So $c\big|_{X_g}$ can be extended to a locally admissible coloring of $gS\cup gSG_t$.
In other words, every $c\in E_{t'}^{(F_n)}\big|_{X_g}$ can be extended to a coloring of $gS$ that can itself be extended into a coloring of $gS\cup gSG_t$. Hence,
\begin{equation*}
\left|E^{(F_n)}_{t'}\big|_{X_g}\right|
\le \left|E^{(g\cdot S)}_{t}\right|
=\left|E^{(S)}_{t}\right|
\le ((1-\varepsilon)\tilde\alpha(\F))^{|S|} \,,
\end{equation*}
where the two last inequalities are respectively consequences of Fact \ref{extendabilityIsTranslationInvariant} and of equation \eqref{contradiction}.
Plugging this inequality in \eqref{eqshearerapplied}, we obtain
\begin{equation}\label{entropyalmostdone}
\left|E_{t'}^{(F_n)}\right|\le\left(\prod_{g\in F_n S^{-1}} ((1-\varepsilon)\tilde\alpha(\F))^{|S|}\right)^{1/|S|}
\le ((1-\varepsilon)\tilde\alpha(\F))^{|F_n S^{-1}|}\,.
\end{equation}
If $(1-\varepsilon)\tilde\alpha(\F)<1$, then $|E^{(F_n)}_{t'}|=0$ which implies $\tilde\alpha(\F)=0$ and concludes the proof. We can assume in the following that $(1-\varepsilon)\tilde\alpha(\F)\ge1$.
Moreover, since $(F_n)_{n\in \mathbb{N}}$ is a Følner sequence,
\begin{align*}
\lim_{n\rightarrow\infty}\frac{|F_n\cdot S^{-1}|}{|F_n|}\le 1+\lim_{n\rightarrow\infty}\frac{|(F_n\cdot S^{-1})\Delta F_n|}{|F_n|}
  =1\,.
\end{align*}
Fix some $0<\varepsilon'<\varepsilon$, then $\frac{1-\varepsilon'}{1-\varepsilon}>1$, and there exists $\varepsilon''>0$
such that $\tilde\alpha(\F)^{\varepsilon''}<\frac{1-\varepsilon'}{1-\varepsilon}$.
Moreover, for $n$ large enough, $|F_n S^{-1}|\le |F_n|(1+\varepsilon'')$.
Together with equation \eqref{entropyalmostdone} this implies that, for $n$ large enough,
\begin{align*}
\left|E_{t'}^{(F_n)}\right|^{1/|F_n|}
\le ((1-\varepsilon)\tilde\alpha(\F))^{(1+\varepsilon'')}
< \tilde\alpha(\F) \left((1-\varepsilon) \tilde\alpha(\F)^{\varepsilon''}\right)<\tilde\alpha(\F)(1-\varepsilon')\,.
\end{align*}
That is,
\begin{align*}
\lim_{n\rightarrow \infty}\left|E_{t'}^{(F_n)}\right|^{1/|F_n|}\le\tilde\alpha(\F)(1-\varepsilon')<\tilde\alpha(\F)
\end{align*}
which contradicts \eqref{growthOfFølner} and finishes our proof.
\end{proof}

We are now ready to prove our theorem.

\begin{proof}[Proof of Theorem \ref{extendablesamegrowth}]
For all $t$, any globally admissible pattern is $t$-extendable and any $(t+1)$-extendable pattern is $t$-extendable, thus for all $S\subset G$,
\begin{equation}
\mathcal{G}_\F^{(S)}\subseteq\ldots \subseteq E^{(S)}_{t+1}\subseteq E^{(S)}_t\subseteq\ldots\subseteq  E^{(S)}_1\subseteq \mathcal{L}_\F^{(S)}\,.
\end{equation}
Moreover, by compactness $\mathcal{G}_\F^{(S)}=\bigcap\limits_{t\ge0}E^{(S)}_t$. Since all the $E^{(S)}_t$ are finite, it implies that for all $S\in G$, there exists $t_S$ such that,
\begin{equation*}
\mathcal{G}_\F^{(S)}=\ldots = E^{(S)}_{t_S}\subseteq\ldots\subseteq  E^{(S)}_1\subseteq \mathcal{L}_\F^{(S)}\,.
\end{equation*}
From Lemma \ref{limToInf}, for all $S$,
$\left|\mathcal{G}_\F^{(S)}\right|= \left|E^{(S)}_{t_S}\right|\ge \tilde\alpha(\F)^{|S|}\,.$
Hence, for all $n$,
\begin{equation*}
\mathbf{G}_\F^{(n)}\ge\tilde\alpha(\F)^n\,,
\end{equation*}
which finally implies
\begin{equation*}
\alpha(\F)=\lim\limits_{n\rightarrow\infty}\left(\mathbf{G}_\F^{(n)}\right)^{1/n}\ge\tilde\alpha(\F)
\end{equation*}
as desired.  
\end{proof}

The amenability condition in Theorem \ref{extendablesamegrowth} is necessary. In deed, we prove in Theorem \ref{amenableNecessary} that for every countable non-amenable group there exists a set of forbidden patterns such that $\tilde\alpha(\mathcal{F})>\alpha(\mathcal{F})$. 
This gives an alternative characterization of amenable countable groups: a countable group is amenable if and only if for any set of forbidden patterns $\mathcal{F}$, we have 
$\alpha(\mathcal{F})=\tilde\alpha(\mathcal{F})\,.$ 
The lower bounds on $\tilde\alpha(\mathcal{F})$ that we provide in the remainder of this article are still meaningful for non-amenable group since, by compactness, $\tilde\alpha(\mathcal{F})\ge1$ implies $\alpha(\mathcal{F})\ge1$.

\begin{theorem}\label{amenableNecessary}
Let $G$ be a countable non-amenable group. Then there exists a set of forbidden patterns $\mathcal{F}$ such that 
$$\tilde\alpha(\mathcal{F})>\alpha(\mathcal{F})\,.$$
\end{theorem}
\begin{proof}
The negation of the existence of a Følner sequence of $G$ implies that there exists a finite subset $S$ of $G$ and $\varepsilon>0$ such that for all $F$, there exists $s\in S$,
$\frac{|sF\Delta F|}{|F|}>\varepsilon\,,$
which implies $|(sF)\setminus F|>\frac{\varepsilon}{2} |F|$.
Let $S'=S\cup\{\mathbf{1}_G\}$ where $\mathbf{1}_G$ is the neutral element of $(G,\cdot)$, then for all $F$, 
$$|(S'\cdot F)\setminus F|=|S\cdot F\setminus F|>\frac{\varepsilon}{2} |F|\,.$$ Since $S'$ is finite, we have for all $F$,
\begin{equation}\label{sizeDiff}  
|\{f\in F:S'\cdot f\not\subseteq F\}|>\frac{\varepsilon}{2|S'|} |F|\,.
\end{equation}

We are now ready to construct our set of forbidden patterns. 
Consider the alphabet $\mathcal{A}=\{0,1\}$ and the set of forbidden patterns
\begin{equation*}
\mathcal{F}= \left\{p\in\mathcal{A}^{S'} : p(\mathbf{1}_G)=0\right\}\,.
\end{equation*}
The only configuration in the shift $\mathcal{X}_\mathcal{F}$ is the constant configuration where every element receives $1$. Thus, $\alpha(\mathcal{F})=1$.
On the other hand, for every finite subset $F\subseteq G$, let $\mathcal{B}:=\{f\in F :S'\cdot f\subseteq F\}$. Any coloring $c$ of $F$ such that for all $x\in \mathcal{B}$, $c(x)=1$ are locally admissible, so there are at least  $2^{|F-\mathcal{B}|}>2^{|F|\varepsilon/(2|S'|)}$ locally admissible coloring of $F$ (where the inequality is a direct consequence of equation \eqref{sizeDiff}). 
This implies $\mathbf{L}_\F^{(n)}> 2^{n\varepsilon/(2|S'|)}$, that is,
$$\tilde\alpha(\mathcal{F})> 2^{\varepsilon/(2|S'|)}>1=\alpha(\mathcal{F})\,,$$
as desired.
\end{proof}

\section{Lower bound on the growth rate in the general case}\label{generalbound}
Now that we showed that the local growth is identical to the global growth, we can provide our main result to lower-bound the local growth.

\begin{lemma}\label{lemmacounting}
  Let $\F\subseteq\A^+_G$ and $\beta$ be a positive real number such that
  \begin{equation}\label{thmHyp}
|\A|-\sum_{f\in\F}|f|\beta^{1-|f|}\ge \beta\,.
  \end{equation}
  Then for all finite sets $S\in G$ and $s\in G\setminus S$,
\begin{equation*}
 |\mathcal{L}_\F^{(S\cup\{s\})}|\ge\beta|\mathcal{L}_\F^{(S)}|\,.
\end{equation*}
\end{lemma}
\begin{proof}
We proceed by induction on $S$. Let $S$ be such that for all $X\subseteq S$ and for all $x\in S\setminus X$
\begin{equation*}
 \left|\mathcal{L}_\F^{(X\cup\{x\})}\right|\ge\beta\left|\mathcal{L}_\F^{(X)}\right|\,.
\end{equation*}
For all $R\subseteq S$ we can use our hypothesis inductively and we obtain
\begin{equation}\label{IHPmain}
\left|\mathcal{L}_\F^{(S\setminus R)}\right|\le\frac{\left|\mathcal{L}_\F^{(S)}\right|}{\beta^{|R|}}\,.
\end{equation}
An \emph{extension} of a pattern $p\in \mathcal{L}_\F^{(S)}$ is a pattern $p'\in \A^{S\cup\{s\}}$, such that $p'|_{S}=p$. The number of extension is $|\A|\cdot|\mathcal{L}_\F^{(S)}|$. Let $B$ be the set of extensions that are not locally admissible, then
\begin{equation}\label{forbidenextansions}
  |\mathcal{L}_\F^{(S\cup\{s\})}|\ge |\A|\cdot|\mathcal{L}_\F^{(S)}| - |B|\,.
\end{equation}
For all $f\in \F$, let $B_f$ be the set of extensions $p$ such that $f$ appears in $p$.
Then $B=\bigcup_{f\in \F} B_f$ and
\begin{equation}\label{sizeofB}
|B|\le \sum_{f\in \F}|B_f|\,.
  \end{equation}
For any $f\in \F$ and any $p\in B_f$, the pattern $f$ appears in $p$. For any $g\in G$, we let $B_{f,g}$ be the set of extensions such that $f$ appears at position $g$. 

Fix $g\in G$ and let $T=\{gx: x\in \operatorname{supp}(f)\}$. Since  $p|_{S}\in \mathcal{L}_\F^{(S)}$, we know that $s\in T$.
If $f$ occurs at position $g$ in a pattern $p$, then $p|_{T}$ is uniquely determined by $f$ and $g$.
Moreover, for all $p\in B_{f,g}$, we have $p|_{(S\cup \{s\})\setminus T}\in\mathcal{L}_\F^{((S\cup \{s\})\setminus T)}$, hence $|B_{f,g}|\le|\mathcal{L}_\F^{((S\cup \{s\})\setminus T)}|$. By equation \eqref{IHPmain},
\begin{equation*}|B_{f,g}|\le \frac{|\mathcal{L}_\F^{(S)}|}{\beta^{|T\setminus\{s\}|}}
\le \frac{|\mathcal{L}_\F^{(S)}|}{\beta^{|f|-1}}\,.\end{equation*}
There are at most $|T|=|f|$ possible values of $g$ such that $s\in T$. It implies that
\begin{equation}\label{boundonBF}
|B_f|\le |f|\frac{|\mathcal{L}_\F^{(S)}|}{\beta^{|f|-1}}\,.
\end{equation}
We can finally use this together with equations \eqref{forbidenextansions} and \eqref{sizeofB} to obtain
\begin{equation*}
  |\mathcal{L}_\F^{(S\cup\{s\})}|\ge |\A|\cdot|\mathcal{L}_\F^{(S)}| - \sum_{f\in \F}|B_f|
  \ge|\mathcal{L}_\F^{(S)}| \left(|\A|-\sum_{f\in \F} |f|\beta^{1-|f|}\right)\,.
\end{equation*}
We can finally apply our theorem hypothesis \eqref{thmHyp} to deduce $|\mathcal{L}_\F^{(S\cup\{s\})}|\ge \beta|\mathcal{L}_\F^{(S)}|$ as desired.
\end{proof}

We showed that adding an element to the support multiplies the number of locally admissible patterns by $\beta$. Moreover, the empty pattern is always locally admissible, so we deduce the following corollary.
\begin{corollary}\label{growthofL}
  Let $\F\subseteq\A^+_G$ and $\beta$ be a positive real number such that
  \begin{equation*}|\A|-\sum_{f\in\F}|f|\beta^{1-|f|}\ge \beta.\end{equation*}
  Then for all finite sets $S\in G$,
 \begin{equation*} \mathbf{L}^{(S)}_\F\ge\beta^{|S|}\,.\end{equation*}
\end{corollary}

In particular, the previous Corollary implies that, for any finite support, there exists at least one locally admissible configuration. The usual compactness argument immediately implies the following Corollary.
\begin{corollary}
Let $G$ be a countable group, $\F\subseteq\A^+_G$ and $\beta$ be a positive real number such that
  \begin{equation*}|\A|-\sum_{f\in\F}|f|\beta^{1-|f|}\ge \beta.\end{equation*}
  Then
  \begin{equation*}X_\F\not=\emptyset\,.\end{equation*}
\end{corollary}

If instead of a simple compactness argument, we apply Theorem \ref{extendablesamegrowth} we obtain a lower bound on the global complexity of the subshift.
Recall that the entropy of any subshift $X$ over a countable amenable group $G$ is given by $h(X)=\log \alpha(X)$.

\begin{theorem}\label{mainth}
  Let $G$ be a countable amenable group, $\F\subseteq\A^+_G$ and $\beta$ be a positive real number such that
  \begin{equation*}|\A|-\sum_{f\in\F}|f|\beta^{1-|f|}\ge \beta.\end{equation*}
  Then \begin{equation*}\alpha(X_\F)\ge \beta \quad  \quad  \text{ and }\quad  \quad  h(X_\F)\ge\log \beta\,.\end{equation*}
\end{theorem}
\begin{proof}
Corollary \ref{growthofL} implies that $\mathbf{L}^{(n)}_\F\ge\beta^{n}$ for all $n\ge0$. Hence,
$\tilde\alpha(\F)\ge\beta$ and Theorem \ref{extendablesamegrowth} implies $\alpha(\F)\ge\beta$.
\end{proof}

\section{Applications}
It is hard to compare our conditions to the conditions given in \cite{abb2019} and \cite{Bernshteyn}. It seems that in the most general context none of them is weaker than the other one. However, in applications, our condition seems to be more general and is easier to optimize. In particular, we obtain better bounds on all of their applications. Lemma 2.2 of \cite{abb2019} is asymmetric in the sense that they associate a different value $x(g)$ to each vertex. It does not seem to be helpful, because of the symmetries of groups. Although it seems possible to prove an asymmetric version of Theorem \ref{mainth} where a different value $\beta(v)$ is associated to each vertex $v$, it is not necessary for our applications.

\subsection{Strongly aperiodic shift}
A configuration $X\in A^G$ has \emph{period} $g$, if for all $x\in G$, $X(x)= X(gx)$. A configuration is \emph{strongly aperiodic} if its only period is $1_G$ and a subshift is strongly aperiodic if all of its configurations are strongly aperiodic. We refer the reader to \cite{abb2019} for more context about strongly aperiodic shifts.

Let $G$ be an infinite countable group and let $(s_i)_{i\ge0}$ be an enumeration of the elements of $G$ such that $s_0=1_G$.
Let $(T_i)_{i\ge1}$ be a sequence of finite subsets of $G$ such that for every $i\ge1$, $T_i\cap \{s_i \cdot x: x\in T_i\} = \emptyset$ and
$|T_i|=C\cdot i$, where $C$ is a constant to be defined later. We can always find such a sequence, since $G$ is infinite.
For all $i\ge1$, we let $\mathcal{P}_i$ be the set of patterns of support $ T_i\cup\{s_i \cdot x: x\in T_i\}$ such that for all $f\in \mathcal{P}_i$, and all $t\in T_i$, $f(t)=f(s_i\cdot t)$. Finally, we let $\mathcal{P}= \bigcup_{i\ge1}\mathcal{P}_i$.

A configuration $x:G\rightarrow \A$ that avoids $\mathcal{P}$ is strongly aperiodic. Indeed, if $x$ has period $p$ then $x$ contains an occurrence of at least a pattern from $\mathcal{P}_i$ with $i$ such that $s_i=p$ which contradicts the definition of $x$.
In \cite[Theorem 2.4]{abb2019}, they deduce from their criterion that if $C\ge17$, then $X_\mathcal{P}$ is non-empty. It implies the existence of a strongly aperiodic subshift over any countable group.

It is easy to see that for all $i\ge1$, $|\mathcal{P}_i|\le 2^{Ci}$ and for all $f\in \mathcal{P}_i$, $|f|=2Ci$. If we can find $\beta$ such that
\begin{equation}\label{eqtosatisfy}
2-\sum_{i\ge1}2Ci \cdot2^{Ci} \cdot\beta^{1-2Ci}\ge \beta
\end{equation}
then we can apply Theorem \ref{mainth}. If $\beta> \sqrt{2}$ it is equivalent to
\begin{equation*}
2- \frac{2^{1+C}\beta^{1+2C}C}{(\beta^{2C}-2^C)^2}\ge\beta
\end{equation*}
which
holds for $C=11$ and $\beta = 1.9$. A more careful analysis implies that there exists such a $\beta$ for all $C\ge11$. In fact, we can show the following stronger result.
\begin{theorem}\label{alsoAperiodic}
  Let $G$ be a countable group, $\F\subseteq\A^+_G$ and $\beta$ be a positive real number such that
  \begin{equation*}|\A|-\sum_{f\in\F}|f|\beta^{1-|f|}>\beta.\end{equation*}
  Then there exists a strongly aperiodic subshift $X$ avoiding $\F$. Moreover, if $G$ is amenable, we have  a strongly aperiodic subshift $X$ avoiding $\F$ and such that
  \begin{equation*}
  \alpha(X)\ge \beta\,.
  \end{equation*}
\end{theorem}
The condition is similar to the condition given in Theorem \ref{mainth}, the main difference being that we require a strict inequality. It means that if we remove some small enough quantity to the left-hand side of the inequality, the inequality still holds.

We defined the set $\mathcal{P}$ of patterns in such a way that any configuration that avoids $\mathcal{P}$ is strongly aperiodic. So we only need to forbid $\mathcal{F}\cup\mathcal{P}$.
In particular, if we add $\mathcal{P}$ to the set of forbidden patterns and if $\beta>\sqrt{2}$ we can choose $C$ marge enough such that the inequality remains strict and in this case, this theorem is a direct corollary of Theorem \ref{mainth}. The case $\beta\le\sqrt{2}$ is slightly more complicated, but it is a direct consequence of the following lemma. 
\begin{lemma}
  Let $\F\subseteq\A^+_G$ and $\beta>1$ be a real number such that
  \begin{equation}\label{thmHyp2}
|\A|-\sum_{f\in\F}|f|\beta^{1-|f|}> \beta\,.
  \end{equation}
  If $C$ is large enough then for all finite set $S\in G$ and $s\in G\setminus S$,
 \begin{equation*} |\mathcal{L}_{\F\cup \mathcal{P}}^{(S\cup\{s\})}|\ge\beta|\mathcal{L}_{\F\cup \mathcal{P}}^{(S)}|\,.\end{equation*}
\end{lemma}
The proof is almost identical to the proof of Lemma \ref{lemmacounting}, except that we also need to forbid $\mathcal{P}$ and we use an extra trick to forbid $\mathcal{P}$ ``more efficiently''.
For any set $S\subseteq G$ and $g\in G$, we let $g\cdot S= \{g\cdot s: s\in S\}$.
\begin{proof}
  We proceed by induction on $S$. Let $S$ be such that for all $X\subseteq S$ and for all $x\in G\setminus X$
\begin{equation*} \left|\mathcal{L}_{\F\cup \mathcal{P}}^{(X\cup\{x\})}\right|\ge\beta\left|\mathcal{L}_{\F\cup \mathcal{P}}^{(X)}\right|\,.\end{equation*}
For any $R\subseteq S$, we can use our hypothesis inductively, and we obtain
\begin{equation}\label{IHPmain2}
\left|\mathcal{L}_{\F\cup \mathcal{P}}^{(S\setminus R)}\right|\le\frac{\left|\mathcal{L}_{\F\cup \mathcal{P}}^{(S)}\right|}{\beta^{|R|}}\,.
\end{equation}
An \emph{extension} of a pattern $p\in \mathcal{L}_{\F\cup \mathcal{P}}^{(S)}$ is a pattern $p'\in \A^{S\cup\{s\}}$, such that $p'|_{S}=p$. The number of extension is $|\A|\cdot|\mathcal{L}_{\F\cup \mathcal{P}}^{(S)}|$. Let $B$ be the set of extensions that are not locally admissible, then
\begin{equation}\label{forbidenextansions2}
  |\mathcal{L}_{\F\cup \mathcal{P}}^{(S\cup\{s\})}|\ge |\A|\cdot|\mathcal{L}_{\F\cup \mathcal{P}}^{(S)}| - |B|\,.
\end{equation}
For all $f\in \F$, let $B_f$ be the set of extension $p$ such that $f$ appears in $p$.
For all $i\ge1$, let $B'_i$ be the set of extensions $p$ such that there exists $f\in \mathcal{P}_i$ and $f$ appears in $p$.
Then $B=\bigcup_{f\in \F} B_f\cup \bigcup_{i\ge1} B'_i$ and
\begin{equation}\label{sizeofB2}
|B|\le \sum_{f\in \F}|B_f|+\sum_{i\ge1} |B'_i|\,.
  \end{equation}
  With the exact same argument as in Lemma \ref{lemmacounting}, we can show that for all $f\in \F$
\begin{equation*}
|B_f|\le |f|\frac{|\mathcal{L}_{\F\cup \mathcal{P}}^{(S)}|}{\beta^{|f|-1}}\,.
\end{equation*}

For all $c\in B'_i$, there exists $g\in G$ such that for all $t\in T_i$, $c(g\cdot t)= c(g\cdot s_i\cdot t)$ and $s\in g\cdot T_i$ or $s\in  g\cdot s_i\cdot T_i$. If $s\in g\cdot T_i$ (resp., $  g\cdot s_i\cdot T_i$), then $c$ is uniquely determined by $g$ and $c|_{g\cdot s_i\cdot T_i}$ (resp., $c|_{s_i\cdot T_i}$) which belongs to $\mathcal{L}_{\F\cup \mathcal{P}}^{(S\cup\{s\}\setminus(g\cdot s_i\cdot T_i))}$ (resp., $\mathcal{L}_{\F\cup \mathcal{P}}^{(S\cup\{s\}\setminus(s_i\cdot T_i))}$). Since there are $2\cdot |T_i|= 2iC$ possible choices for $g$, we can use equation \eqref{IHPmain2} to obtain
\begin{equation*}|B'_i|\le 2iC \frac{|\mathcal{L}_{\F\cup \mathcal{P}}^{(S)}|}{\beta^{|T_i|-1}}\le 2iC \frac{|\mathcal{L}_{\F\cup \mathcal{P}}^{(S)}|}{\beta^{iC-1}}\,.\end{equation*}

We can use the two previous bounds together with equations \eqref{forbidenextansions2} and \eqref{sizeofB2} to obtain
\begin{align*}
  |\mathcal{L}_{\F\cup \mathcal{P}}^{(S\cup\{s\})}|
  &\ge |\A|\cdot|\mathcal{L}_{\F\cup \mathcal{P}}^{(S)}| - \sum_{f\in \F}|B_f|-\sum_{i\ge1} |B'_i|\\
  &\ge|\mathcal{L}_{\F\cup \mathcal{P}}^{(S)}| \left(|\A|-\sum_{f\in \F} |f|\beta^{1-|f|}-\sum_{i\ge1} \frac{2iC}{\beta^{iC-1}}\right) \,.
\end{align*}
Since $\beta>1$, $\sum_{i\ge1} \frac{2iC}{\beta^{iC-1}}=\frac{2\beta^{C+1}C}{(\beta^C-1)^2}$ and $\lim\limits_{C\rightarrow\infty}\frac{2\beta^{C+1}C}{(\beta^C-1)^2}=0$, so  we can make $\frac{2\beta^{C+1}C}{(\beta^C-1)^2}$ arbitrarily small by choosing $C$ large enough ($C$ only depends on $|A|$, $\F$ and $\beta$).
Together with hypothesis \eqref{thmHyp2}, this implies that there exists $C$ such that $|\A|-\sum_{f\in \F} |f|\beta^{1-|f|}-\sum_{i\ge1} \frac{2iC}{\beta^{iC-1}}>\beta$. Using this in the previous equation yields
\begin{equation*}
|\mathcal{L}_{\F\cup \mathcal{P}}^{(S\cup\{s\})}|\ge|\mathcal{L}_{\F\cup \mathcal{P}}^{(S)}| \beta
\end{equation*}
as desired.
\end{proof}

Applying Theorem \ref{alsoAperiodic} to $\F=\emptyset$ yields the following corollary.
\begin{corollary}
For all infinite countable amenable group $G$ and all $\varepsilon>0$, there exists a subshift $X\subseteq \{0,1\}^G$ such that $X$ is strongly aperiodic and $\alpha(X)\ge 2-\varepsilon$ (or equivalently $h(X)\ge\log (2-\varepsilon)$).
\end{corollary}
This result is optimal, in the sense that we cannot find such an $X$ with $\alpha(X)=2$.
Indeed, for any subshift $X\subseteq \{0,1\}^G$, if $\alpha(X)=2$ then $X$ is the full shift, that is, $X=\{0,1\}^G$ and $X$ is not strongly aperiodic.

\subsection{Non-repetitive subshift}
For any graph $G=(V,E)$ and coloring $c:V\rightarrow A$ of the vertices of $G$ with the color set $A$, we say that a path $p=v_1v_2\ldots v_{2n}$ is \emph{repetitively colored} if $c(v_i)=c(v_{n+i})$ for all $i\in \{1,\ldots, n\}$. If there is no such repetitively colored path, then we say that $c$ is a nonrepetitive coloring of $G$. In \cite{abb2019}, they show that for any group $G$ generated by a finite set $S$ the undirected right Cayley graph can be nonrepetitively colored with $2^{19}|S|^2$ colors. A direct application of Theorem \ref{mainth} yields a better bound on the number of colors required.
\begin{lemma}\label{nonrepcol}
 For any countable group $G$ generated by a finite set $S$, the undirected right Cayley graph can be nonrepetitively colored with $4|S|^2+16|S|^{5/3}$ colors. Moreover, if $G$ is amenable and  $\mathcal{X}$ is the subshift of nonrepetitive colorings of $G$, then  \begin{equation*}\alpha(X)\ge4|S|^2+12|S|^{5/3}\,.\end{equation*}
\end{lemma}
\begin{proof}
Let $\A$ be an alphabet of size $4|S|^2+16|S|^{5/3}$.
We let $\F$ be the set of repetitively colored paths starting from $1_G$. The number of paths starting from $1_G$ with $2n$ vertices is at most $(2|S|)^{2n-1}$. For each support, the number of corresponding forbidden patterns is $|\A|^n$, since the colors of the first half of the path impose the colors of the second half. Thus, for every $n$, there are $(2|S|)^{2n-1}|\A|^n$ forbidden patterns of size $2n$. We only need to find $\beta$ such that
 \begin{equation*}
 |\A|-\sum_{n\ge1}2n \frac{(2|S|)^{2n-1}|\A|^n}{\beta^{2n-1}}\ge \beta
 \end{equation*}
 holds and we can apply Theorem \ref{mainth} to conclude. In particular with $\beta=4|S|^2+12|S|^{5/3}$, using the fact that $|S|>1$
 \begin{align*}
\sum_{n\ge1}2n \frac{(2|S|)^{2n-1}|\A|^n}{\beta^{2n-1}}
&=\frac{2\beta}{2|S|}\sum_{n\ge1}n \left(\frac{(2|S|)^{2}|\A|}{\beta^2}\right)^n\\
&\ge 4|S|\sum_{n\ge1}n \left(\frac{1+4|S|^{-1/3}}{(1+3|S|^{-1/3})^2}\right)^n\\
&=4|S| \frac{(1+3|S|^{-1/3})^2(1+4|S|^{-1/3})}{|S|^{-2/3}(2+9|S|^{-1/3})^2}\\
&\le4|S|^{5/3}
 \end{align*}
 which concludes our proof.
\end{proof}
In fact, it is known that for any graph of degree at most $\Delta$ the $\Delta^2+\frac{3}{2^{-2/3}}\Delta^{5/3}+O(\Delta^{4/3})$ color suffices to obtain a nonrepetitive coloring \cite{woodSurvey}.
Since our Cayley graph has maximum degree at most $2|S|$, it implies a slightly better bound than the one in Lemma \ref{nonrepcol}.
The best bound is obtained by the same counting argument that we used in Lemma \ref{lemmacounting}, with one more extra trick specific to nonrepetitive colorings.

\subsection{Some sets of sizes of forbidden patterns that imply nonempty subshift}
With his conditions, Miller showed amongst other things that over $G=\mathbb{Z}$ if there is in  $\F$ at most one connected pattern of each size in $\{5, 6, 7,\ldots \}$ and no other pattern then $X_\F$ is nonempty \cite[Corollary 2.2]{MillerSubshift}. We will improve this result in Theorem \ref{improvemiller} by providing a positive lower bound on the entropy of the subshifts instead of simply stating non-emptiness. Here we provide a generalized version of these results that hold for any countable group.
\begin{theorem}\label{generalizemiller}
Let $G$ be a countable group.
Assume that $\F\subset\A_G^+$ is a set of patterns that contains at most one pattern of each size and let $L = \{|p| : p \in \F\}$. If
\begin{enumerate}
  \item $|\A| = 2$ and $L \subseteq \{10,11,12\ldots \}$, then $X_\F$ is non-empty and if moreover $G$ is amenable, then $h(X_\F)\ge\log \alpha_0$ where $\alpha_0\approx1.94$ is the largest root of the polynomial $x^{11}-4x^{10}+5x^9-2x^8+10x-9$,
  \item $|\A| = 3$ and $L \subseteq \{4, 5, 6,\ldots \}$, then $X_\F$ is non-empty and if moreover $G$ is amenable, then $h(X_\F)\ge\log\alpha_1$ where $\alpha_1\approx 2.51$ is the largest root of $x^5-5x^4+7x^3-3x^2+4x-3$,
  \item $|\A| = 4$ and $L \subseteq \{3, 4,5\ldots \}$, then $X_\F$ is non-empty and if moreover $G$ is amenable,  then $h(X_\F)\ge\log\alpha_2$ where $\alpha_2\approx 3.65$ is the largest root of $x^3-4x^2+x+1$,
  \item $|\A| = 5$ and $L \subseteq \{2, 3,4,\ldots \}$, then $X_\F$ is non-empty and if moreover $G$ is amenable,  then $h(X_\F)\ge\log \frac{5+\sqrt{13}}{2}\approx \log 4.30$,
  \item $|\A| = 6$ and $L \subseteq \{1, 2, 3,\ldots \}$, then $X_\F$ is non-empty and if moreover $G$ is amenable,  then $h(X_\F)\ge\log \frac{5+\sqrt{13}}{2}\approx \log 4.30$.
\end{enumerate}
\end{theorem}
\begin{proof}
  For each of these cases, we apply directly Theorem \ref{mainth}. For instance, for case 1., we only need to verify that $2-\sum_{i\ge10} i\alpha_0^{1-i}\ge\alpha_0$. Since $\alpha_0>1$, it is equivalent to
  \begin{equation*}
  2-\frac{10\alpha_0-9}{\alpha_0^8(\alpha_0-1)^2}\ge\alpha_0
  \end{equation*} 
  which is equivalent to $0\ge \alpha_0^{11}-4\alpha_0^{10}+5\alpha_0^9-2\alpha_0^8+10\alpha_0-9$ which holds by hypothesis.
  
   The other cases are all similar.
\end{proof}

\subsection{Nonapplicability to subshifts of subexponential complexity}\label{onlyexp}
 Let \begin{equation*}
 g\colon \begin{array}{lll}
          \mathbb{R}_{>0}&\rightarrow&\mathbb{R} \\
          x&\mapsto& x+\sum_{f\in\F}|f|x^{1-|f|} \,.
        \end{array}
         \end{equation*}
If $\F$ contains at least one pattern of size at least $2$, then there exists some $\varepsilon>0$ such that $g$ is decreasing over $]0,1+\varepsilon]$ (since the derivative $g'$ is $<0$ over $]0,1]$). Hence, if there exists a $\beta$ solution of \eqref{thmHyp}, then there exists a $\beta>1$ solution of \eqref{thmHyp}.
Thus whenever we can apply Theorem \ref{mainth}, it implies that the complexity is exponential (equivalently the subshift has positive entropy). Similarly,  
Theorem \ref{mainth} is useless for subshifts of subexponential complexity (or equivalently, for subshifts of entropy $0$).

\section[Connected support over Z]{Connected support over $\mathbb{Z}$}
Over $\mathbb{Z}$ we say that a pattern $p$ is \emph{connected} if its support is connected, that is there exists integers $i< j$ such that $\operatorname{supp}(p)=\{i, i+1,i+2,\ldots, j-1, j\}$. In this context, connected patterns are usually called \emph{words} or \emph{factors}, but we use ``pattern'' for consistency with the rest of the article.  If the set of forbidden patterns is a set of connected patterns, then we have a slightly stronger version of Lemma \ref{lemmacounting}.

\begin{lemma}\label{lemmacountingZ}
  Let $\F\subseteq\A^+_\mathbb{Z}$ be a set of connected patterns and $\beta$ be a positive real number such that
  \begin{equation*}|\A|-\sum_{f\in\F}\beta^{1-|f|}\ge \beta\,.\end{equation*}
  Then for all integer $i< j$ and $s\in \{i-1,j+1\}$,
 \begin{equation*} |\mathcal{L}_\F^{(\{i,i+1,\ldots, j\}\cup\{s\})}|\ge\beta|\mathcal{L}_\F^{(\{i,i+1,\ldots, j\})}|\end{equation*}
\end{lemma}
\begin{proof}
 The proof is almost identical to the proof of Lemma \ref{lemmacounting} so we only explain how to adapt it.
 The main difference is when we bound the size of $B_f$. In this case, we know that $f$ contains $s$, but since $s$ is the left end (or the right end) of $\{i,i+1,\ldots, j\}\cup\{s\}$, there is only one possible position for the occurrence of $f$. Moreover, since $f$ is connected, so is $\{i,i+1,\ldots, j\}\cup\{s\} \setminus \operatorname{supp}(f)$. Then equation \eqref{boundonBF} becomes
 \begin{equation*}
  |B_f|\le\frac{|\mathcal{L}_\F^{(S)}|}{\beta^{|f|-1}}\,.
\end{equation*}
The rest of the proof is identical and leads to the desired result.
\end{proof}

We could use a similar idea in higher dimensions, but the gain is much smaller. The idea is that the induction can be done over the sets $S$ such that $\mathbb{Z}^d\setminus S$ is connected, so whenever we add an element to $S$ we now have some restrictions on where a forbidden pattern can appear. For instance, in dimension $2$ if all the forbidden patterns are rectangles, then the condition on $\beta$ becomes
  \begin{equation*}|\A|-\sum_{f\in\F}\max(\operatorname{height}(f), \operatorname{width}(f))\beta^{1-|f|}\ge \beta\,.\end{equation*}

Using Theorem \ref{extendablesamegrowth}, we obtain the following simple Corollary of Lemma \ref{lemmacountingZ}.
\begin{corollary}\label{resultZ2}
  Let $\F\subseteq\A^+_\mathbb{Z}$ be a set of connected patterns and $\beta$ be a positive real number such that
  \begin{equation}\label{myCondition}
|\A|-\sum_{f\in\F}\beta^{1-|f|}\ge \beta\,.    
  \end{equation}
  Then $\alpha(X_F)\ge\beta$ and $h(X_\F)\ge \log \beta$.
\end{corollary}
Since the conditions given in \cite{doublepat} are more general than the ones given in Lemma \ref{lemmacountingZ}, we could also deduce a more general version of Corollary \ref{resultZ2}. Let us also mention that the result of Miller has identical conditions, but the conclusion only implies the non-emptiness of the subshift \cite{MillerSubshift}. The ideas behind our proof and the proof from \cite{MillerSubshift,pavlov} share some similarities, which explains the similarities of the conditions. The main difference lies in the fact that they consider the number of possible extensions of a word instead of looking at the suffixes of the words (intuitively, they look ahead at what could go wrong when one tries to extend the word further, while we look at the past to see what could have gone wrong when building the current word).

\section[Applications of the Corollary]{Applications of Corollary \ref{resultZ2}}
\subsection[{Comparison with [10, Theorem 4.1]}]{Comparison with \cite[Theorem 4.1]{pavlov}}
In \cite{pavlov}, Pavlov gave a similar but weaker result that can be restated as follows (by making the replacement $c=\beta^{-1}$ in his result).
\begin{theorem*}[{\cite[Theorem 4.1]{pavlov}}]
  Let $\F\subseteq\A^+_\mathbb{Z}$ be a set of connected patterns and $\beta$ be a positive real number and $k\ge1$ and integer such that
  \begin{equation}\label{pavlovcondition}
|\A|-\sum_{f\in\F}\beta^{1-|f|}> \beta+k-1.
  \end{equation}
  Then $h(X_\F)\ge \log k$.
\end{theorem*}
In order to apply Corollary \ref{resultZ2} optimally, one needs to find the largest $\beta$ such that \eqref{myCondition} holds. Optimal use of Pavlov's result implies finding the largest integer $k$ such that there exists $\beta$ such that equation \eqref{pavlovcondition} holds. The first is usually slightly easier to optimize than the second.

More importantly, his result implies strictly weaker bounds than what can be achieved by  Corollary \ref{resultZ2}.
We say that a set $\F\subseteq\A^+_\mathbb{Z}$ is nontrivial if it contains at least one pattern of support at least $2$. If $\F$ is trivial then the entropy of $X_\F$ is $\log (|\A|- |\F|)$, but in every other case, Corollary \ref{resultZ2} always provides a strictly larger lower-bound on the topological entropy than \cite[Theorem 4.1]{pavlov}.
\begin{lemma}
  For any non-trivial set of connected patterns $\F$, if $(k,\beta)$ is a solution of equation \eqref{pavlovcondition}, then there exists $\beta'>k$ that is solution of equation \eqref{myCondition}.
\end{lemma}
\begin{proof}
Let $g:x\mapsto x +\sum_{f\in\F}x^{1-|f|}$. The pair $(k,\beta)$ is solution of equation \eqref{pavlovcondition} if and only if
$|\A|\ge g(\beta)+k-1$ and
$\beta'$ is solution of equation \eqref{myCondition} if and only if $|\A|\ge g(\beta')$.

Since $\F$ is non-trivial, one easily verifies that $g$ is decreasing over $]0,1[$ (because the derivative $g'(x)\le1-x^{-2}$ over $]0,1[$). Hence, if $(k,\beta)$ is solution of $|\A|\ge g(\beta)+k-1$, there exists $\beta''\ge1$ such that
$|\A|\ge g(\beta'')+k-1$.

For all $x,y\ge1$, \begin{equation*}g(x+y)=x+y+\sum_{f\in\F}(x+y)^{1-|f|}<x+y+\sum_{f\in\F}x^{1-|f|}=g(x)+y\,.\end{equation*}
Applying this to $|\A|\ge g(\beta'')+k-1$ implies $|\A|> g(\beta''+k-1)$. Since this last inequality is strict and $g$ is continuous, there exists $\varepsilon>0$ such that $|\A|> g(\beta''+k-1+\varepsilon)$. Let $\beta'=\beta''+k-1+\varepsilon\ge k+\varepsilon$, then, by the previous equation, $\beta'$ is a solution of \eqref{myCondition} as desired.
\end{proof}
Whenever \cite[Theorem 4.1]{pavlov} can be used to provide a lower bound on the entropy of a shift, Corollary \ref{resultZ2} provides a strictly larger lower-bound (and experimentally the gain is often not negligible). Since \cite[Theorem 4.1]{pavlov} is used multiple times in \cite{pavlov}, we could improve the conditions of a few other results of \cite{pavlov} by simply using our condition. We give one example with \cite[Theorem 7.1]{pavlov}. For any set of forbidden pattern $\F$, let $\F_n=\{f\in \F: |f|=n\}$.

\begin{theorem*}[{\cite[Thm. 7.1]{pavlov}}]  
Let $\F\subseteq\A^+_\mathbb{Z}$ be a set of connected patterns.
  If $\sum\limits_{n\ge1} |F_n|\left(\frac{3}{|\A|}\right)^n<\frac{1}{5}$, then $h(X_\F)>\log(\frac{3|\A|}{4})$.
\end{theorem*}
 A simple application of Corollary \ref{resultZ2} improves this Theorem considerably.
\begin{theorem}
Let $\F\subseteq\A^+_\mathbb{Z}$ be a set of connected patterns.
\begin{enumerate}
  \item   If $\sum\limits_{n\ge1} |\F_n|\left(\frac{3}{|\A|}\right)^n<\frac{1}{5}$, then $h(X_\F)\ge\log\left(\frac{14|\A|}{15}\right)$.
  \item    If $\sum\limits_{n\ge1} |\F_n|\left(\frac{6}{5|\A|}\right)^n<\frac{1}{5}$, then $h(X_\F)\ge \log\left(\frac{5|\A|}{6}\right)$.
  \item    If $\sum\limits_{n\ge1} |\F_n|\left(\frac{3}{|\A|}\right)^n<\frac{3}{4}$, then $h(X_\F)\ge \log\left(\frac{3|\A|}{4}\right)$.
\end{enumerate}
\end{theorem}
\begin{proof}
We apply Corollary \ref{resultZ2}. We only detail the computations of the first case since the two remaining cases are similar.

For 1., we use $\beta= \frac{14|\A|}{15}$ and use the fact that since $\beta^{-1}<\frac{3}{|\A|}$,
\begin{align*}
\sum_{f\in\F}\beta^{1-|f|}&=\sum_{n\ge1} |\F_n|\beta^{1-n}\\
&\le \sum\limits_{n\ge1} |\F_n|\left(\frac{3}{|\A|}\right)^{n-1}
=\frac{|\A|}{3}\sum\limits_{n\ge1} |\F_n|\left(\frac{3}{|\A|}\right)^{n}
\le\frac{|\A|}{15}\,.
\end{align*}
We have $|\A|-\sum_{f\in\F}\beta^{1-|f|} =\beta$ so we can apply Corollary \ref{resultZ2}.

For 2., we use $\beta = \frac{5|\A|}{6}$.

For 3., we use $\beta = \frac{3|\A|}{4}$.
\end{proof}
Statements 1. and 2. have stronger conclusions and statements 2. and 3. have weaker conditions than \cite[Thm. 7.1]{pavlov}.

\subsection[{Improving the conclusions of [8]}]{Improving the conclusions of \cite{MillerSubshift}}
As already stated, Miller showed that under the conditions of Corollary \ref{resultZ2} the subshift is non-empty \cite{MillerSubshift}. He then gave a few applications. We can apply Corollary \ref{resultZ2} to each of these results to obtain a lower bound on the entropy of the subshift.

\begin{theorem}\label{improvemiller}
Assume that $\F\subset\A_\mathbb{Z}^+$ is a set of connected patterns that contains at most one pattern of each length and let $L = \{|p| : p \in \F\}$. If
\begin{enumerate}
  \item $|\A| = 2$ and $L \subseteq \{5, 6, 7,\ldots \}$, then $h(X_\F)\ge\log \alpha_1$ where $\alpha_1\approx1.755$ is the largest root of $x^3-2x^2+x-1$,
  \item $|\A| = 2$ and $L \subseteq \{4, 6, 8,\ldots \}$, then $h(X_\F)\ge\log \frac{1+\sqrt{5}}{2}\approx \log 1.618$,
  \item $|\A| = 3$ and $L \subseteq \{2, 3, 4,\ldots \}$, then $h(X_\F)\ge\log 2$,
  \item $|\A| = 4$ and $L \subseteq \{1, 2, 3,\ldots \}$, then $h(X_\F)\ge\log 2$.
\end{enumerate}
\end{theorem}
The proof is a direct application of Corollary \ref{resultZ2}. The result from \cite{MillerSubshift} only implied the non-emptiness of these subshifts with no lower bound on the entropy other than $h(X_\F)\ge0$.

 \subsection{Kolmogorov complexity}
 Roughly speaking, the Kolmogorov complexity $C(x)$ of a string $x$ is the size of the shortest program that outputs this string. We will not provide a presentation of Kolmogorov complexity and we redirect the reader to the literature \cite{bookKolmog1,bookkolmogshen}. The only fact that we will use is that for any integer $n$, there are less than $2^n$ strings $x$ such that $C(x)<n$ (see \cite[Theorem 5]{bookkolmogshen} for instance, although it is a direct consequence of the fact that there are less than $2^n$ programs of length less than $n$).
 
 In \cite{MillerSubshift} Miller obtained a new simpler proof of the following result due to Durand \textit{et~Al.}~\cite{tilinghighkolmogorov}.
 \begin{theorem*}[\cite{MillerSubshift,tilinghighkolmogorov}]
Let $d < 1$. There is an $X\in \{0,1\}^\mathbb{Z}$
such that if $\tau\in \{0,1\}_{\mathbb{Z}}^+$ appears in $X$ then $K(\tau)>d |\tau|-O(1)$.
 \end{theorem*}
Once again, instead of simply obtaining the existence of $X$, we can show that there exists a subshift $X$ with this property that has entropy arbitrarily close to the entropy of the full shift.
 \begin{theorem}
Let $d < 1$ and $\beta$ such that $2^d<\beta<2$. Then there exists a constant $C>0$ and a subshift $X\in \{0,1\}^\mathbb{Z}$ such that $h(X)\ge \log(\beta)$ and for all connected pattern $\tau$ that appears in $X$, $K(\tau)>d |\tau|-C$.
 \end{theorem}
 \begin{proof}
   We let $C=1+d+\log\left(\frac{\beta}{(2-\beta)(\beta-2^d)}\right)$.
We let $\F$ be the set of connected patterns $f$ such that $K(f)\le d |f|-C$.
Then for all $n$, \begin{equation*}|\{f\in \F: |f|=n\}|\le 2^{1+dn-C}\,.\end{equation*}
It implies
\begin{equation*}
2-\sum_{f\in \F} \beta^{1-|f|}\ge2-\sum_{i\ge1}2^{1+dn-C} \beta^{1-n}=2-\frac{2^{1-C+d}\beta}{\beta-2^d}=\beta
\end{equation*}
We apply Corollary \ref{resultZ2} to deduce that $h(X_\F)\ge \log \beta$ which concludes our proof.
 \end{proof}

This result is optimal in the sense that the only subshift $X$ with $h(X)=\log 2$ is the full shift so we cannot hope to do better than having $h(X)$ arbitrarily close to $\log 2$. 

It is also optimal in the sense that there is no constant $C$ such that we can avoid all the connected patterns $\tau$ such that $K(\tau)\le|\tau|-C$.\footnote{The case $K(\tau)=d|\tau|+o(1)$ is already implied by our Theorem by simply applying it with some $d'>d$.} Indeed, suppose that there is at least one pattern $u$ avoided, then we find an encoding $K'$ of the remaining patterns such that for all $v$ avoiding $u$, $K(v)= \mathcal{K'}(v)= \log\mathcal{O}\left((2^{|u|}-1)^{\frac{|v|}{|u|}}\right)=o({|v|})$.
This result is generalizable to any countable group instead of $\mathbb{Z}$ by applying Theorem \ref{mainth} instead of Corollary \ref{resultZ2}.

\section*{Acknowledgement}
The author would like to thank Emmanuel Jeandel and Nathalie Aubrun for bringing some of the mentioned results to his attention, Sebastián Barbieri for some helpful discussions and the two referees for helpful remarks.  

The author also thanks Georgii Veprev for noticing a critical mistake on an earlier version of this work.

\end{document}